\documentclass[11pt]{amsart}
\usepackage[
  left=1.15in,
  right=1.15in,
  top=1.0in,
  bottom=1.05in
]{geometry}
\usepackage{amsmath,amsthm,amssymb,mathtools,amsfonts,amsopn,amscd}
\usepackage{bm}
\usepackage{tikz}
\usepackage{tikz-cd}    % commutative diagrams with tikz
\usepackage{rotating}
\usepackage{graphicx}
\usepackage{etoolbox}
\usepackage{enumitem}
\usepackage{longtable}
\usepackage{booktabs} 
\usepackage{tabularx} 
\usepackage{comment}
\usepackage{xcolor}
\usepackage{hyperref}
\usepackage{subfiles}
\usepackage{calc}
\usepackage[numbered]{bookmark}  
\usepackage{mathrsfs}
\usepackage[sc]{mathpazo}
\usepackage{euscript}
\usepackage[normalem]{ulem}
\usepackage[noabbrev,nosort]{cleveref}
\crefname{section}{Section}{Sections}
\Crefname{section}{Section}{Sections}
\crefname{subsection}{Section}{Sections}
\Crefname{subsection}{Section}{Sections}
\crefname{equation}{Equation}{Equations}
\Crefname{equation}{Equation}{Equations}
\crefname{figure}{Figure}{Figures}
\Crefname{figure}{Figure}{Figures}
\crefname{table}{Table}{Tables}
\Crefname{table}{Table}{Tables}
\crefname{thm}{Theorem}{Theorems}
\Crefname{thm}{Theorem}{Theorems}
\crefname{lem}{Lemma}{Lemmas}
\Crefname{lem}{Lemma}{Lemmas}
\crefname{prop}{Proposition}{Propositions}
\Crefname{prop}{Proposition}{Propositions}
\crefname{cor}{Corollary}{Corollaries}
\Crefname{cor}{Corollary}{Corollaries}
\crefname{df}{Definition}{Definitions}
\Crefname{df}{Definition}{Definitions}
\crefname{ex}{Example}{Examples}
\Crefname{ex}{Example}{Examples}
\crefname{rmk}{Remark}{Remarks}
\Crefname{rmk}{Remark}{Remarks}
\crefname{conj}{Conjecture}{Conjectures}
\Crefname{conj}{Conjecture}{Conjectures}
\crefname{clm}{Claim}{Claims}
\Crefname{clm}{Claim}{Claims}

\setlist[enumerate,1]{label=(\roman*)}

\AddToHook{env/lem/begin}{\crefalias{thm}{lem}}
\AddToHook{env/prop/begin}{\crefalias{thm}{prop}}
\AddToHook{env/cor/begin}{\crefalias{thm}{cor}}
\AddToHook{env/df/begin}{\crefalias{thm}{df}}
\AddToHook{env/ex/begin}{\crefalias{thm}{ex}}
\AddToHook{env/rmk/begin}{\crefalias{thm}{rmk}}
\AddToHook{env/clm/begin}{\crefalias{thm}{clm}}
\AddToHook{env/conj/begin}{\crefalias{thm}{conj}}

\definecolor{secondaryColor}{RGB}{0, 0, 170}
\hypersetup{
	colorlinks=true,
	linkcolor=blue,
	urlcolor=secondaryColor,
	citecolor=secondaryColor,
	linktoc=page,
}% 設定超連結參數

\theoremstyle{plain}
\newtheorem{thm}{Theorem}[section]
\newtheorem{lem}[thm]{Lemma}
\newtheorem{prop}[thm]{Proposition}
\newtheorem{cor}[thm]{Corollary}
\newtheorem{conj}[thm]{Conjecture}

\theoremstyle{definition}
\newtheorem{df}[thm]{Definition}

\theoremstyle{remark}
\newtheorem{rmk}[thm]{Remark}

\newcommand{\ZZ}{\mathbb{Z}}
\newcommand{\NN}{\mathbb{N}}
\newcommand{\QQ}{\mathbb{Q}}
\newcommand{\RR}{\mathbb{R}}
\newcommand{\CC}{\mathbb{C}}

\newcommand{\FF}{\mathbb{F}}

\newcommand{\gfk}{\mathfrak{g}}

\newcommand{\Hfk}{\mathfrak{H}}

\newcommand{\Acal}{\mathcal{A}}

\newcommand{\Ical}{\mathcal{I}}

\newcommand{\Scal}{\mathcal{S}}

\newcommand{\Zcal}{\mathfrak{Z}}

\newcommand{\id}{\operatorname{id}}

\let\oldforall\forall
\renewcommand{\forall}{\oldforall \: }
\let\oldexist\exists
\renewcommand{\exists}{\oldexist \: }

\newcommand{\dep}{\operatorname{dep}}
\newcommand{\wt}{\operatorname{wt}}

\newcommand{\ww}[1]{\mathfrak{#1}}

\DeclareMathOperator{\Span}{span}

\newcommand{\ov}[1]{\overline{#1}}
\newcommand{\ant}{{\star-\mathrm{inv}}}

\newcommand{\ev}{\mathrm{ev}}
\newcommand{\sA}{\mathscr{A}}
\newcommand{\Li}{\mathrm{Li}}
\allowdisplaybreaks

\DeclareRobustCommand\longtwoheadrightarrow
     {\relbar\joinrel\twoheadrightarrow}
\makeatletter

\newcommand{\myToC}{{
		\renewcommand{\contentsname}{}
		\@starttoc{toc}{\contentsname}
}}

\patchcmd{\@tocline}
{\hfil}
{\leaders\hbox{\,.\,}\hfil}

\makeatother

\title[A Kaneko--Zagier-Type Conjecture]{Multiple Zeta Values in Positive Characteristic:\\ A Kaneko--Zagier-Type Conjecture}

\author{Hung-Chun Tsui}
\address{(Hung-Chun Tsui) Department of Mathematics, National Tsing Hua University, No. 101, Sec. 2, Guangfu Rd., East Dist., Hsinchu City 300044, Taiwan (R.O.C.)}
\email{hctsui@gapp.nthu.edu.tw}

\date{\today}

\subjclass[2020]{Primary 11M32; Secondary 11M38, 11R58}
\keywords{Function fields, multiple zeta values, finite multiple zeta values}

\begin{document}

\begin{abstract}
In this paper, we study a comparison between $\infty$-adic and finite multiple zeta values in positive characteristic. We construct a surjective comparison map from the algebra of $\infty$-adic multiple zeta values to a quotient of the algebra of finite multiple zeta values. This map factors through the quotient by $\zeta_A(q-1)$, thereby establishing the well-definedness of a Kaneko--Zagier-type conjectural map. We further show that the comparison map is compatible with special values of Carlitz multiple polylogarithms, as well as with the corresponding $\star$-inverse values. As an application, we prove the upper bound predicted by Shi's dimension conjecture for finite multiple zeta values.
\end{abstract}

\maketitle
% \tableofcontents
%%%%%%%%%%%%%%%%%%%%%%%%%%%%%%%%%%%%%%%%
%%%%%%%%%%%%%%%%%%%%%%%%%%%%%%%%%%%%%%%%
\section{Introduction}

\subsection{Classical Background}
\label{sec:classical-background}
We let $\NN$ denote the set of positive integers and $\ZZ$ the set of
integers. Throughout the paper, we adopt the convention that empty sums are zero and empty products are one. We let
\[
\Ical=\{\varnothing\}\cup\bigcup_{r\geq1}\NN^r
\]
denote the set of indices and define 
\[
\dep(\ww{s})=r,
\qquad
\wt(\ww{s})=s_1+\cdots+s_r,\qquad \dep(\varnothing)=\wt(\varnothing)=0
\]
for any $\ww{s}=(s_1,\ldots,s_r)\in\Ical$. 
For an admissible non-empty index
$\ww{s}=(s_1,\ldots,s_r)\in\Ical$, namely $s_1\geq2$, the classical
\emph{multiple zeta value} is defined by
\[
\zeta(\ww{s})
=
\sum_{n_1>\cdots>n_r\geq1}
\frac{1}{n_1^{s_1}\cdots n_r^{s_r}}
\in\RR.
\]
It is known that multiple zeta values admit the stuffle relations arising
from their defining series, as well as the shuffle relations arising from
their iterated integral representations. We put $\Zcal$ for the $\QQ$-algebra generated by all multiple zeta values.

Another deformation of classical multiple zeta values is
provided by \emph{finite multiple zeta values} introduced by Kaneko and Zagier
(see \cite{KZFMZV}; see also \cite{Kan19}). Put
\begin{equation}\label{eq:Acal}
    \Acal
=
\left(\prod_{p}\ZZ/(p)\right)
\bigg/
\left(\bigoplus_{p}\ZZ/(p)\right),
\end{equation}
where $p$ runs through the prime numbers. Note that there is a natural embedding $\QQ\hookrightarrow\Acal$. For a non-empty index $\ww{s}=(s_1,\ldots,s_r)\in\Ical$, the corresponding finite multiple
zeta value is defined by
\[
\zeta_{\Acal}(\ww{s})
=
\left(
\sum_{p>n_1>\cdots>n_r\geq1}
\frac{1}{n_1^{s_1}\cdots n_r^{s_r}}
\bmod p
\right)_p
\in\Acal.
\]
We write $\Zcal_{\Acal}$ for the $\QQ$-algebra generated by all finite
multiple zeta values.

Kaneko and Zagier also introduced the \emph{symmetric multiple zeta values}
$\zeta_{\Scal}(\ww{s})$, $\ww{s}\in\Ical$, which are naturally regarded as elements of
$\Zcal/\zeta(2)\Zcal$. They observed a striking similarity between the
relations satisfied by finite and symmetric multiple zeta values and were
led to the following Kaneko--Zagier conjecture \cite{KZFMZV}.
\begin{conj}[Kaneko, Zagier]\label{conj:kz}
There exists a well-defined $\QQ$-algebra isomorphism
\[
\Phi_{\mathrm{KZ}}:\Zcal_{\Acal}\xrightarrow{\ \sim\ }\Zcal/\zeta(2)\Zcal
\]
such that for any $\ww{s}\in\Ical$,
\[\Phi_{\mathrm{KZ}}(\zeta_{\Acal}(\ww{s}))=\zeta_{\Scal}(\ww{s}).
\]
\end{conj}
It is known by a theorem of Yasuda \cite{Yas16} that the symmetric multiple zeta values
span $\Zcal/\zeta(2)\Zcal$, so the conjectural map is automatically
surjective once it is well-defined. A further connection between these two sides was found by Bachmann,
Takeyama, and Tasaka \cite{BTT18}, who considered \emph{finite multiple harmonic $q$-series}
at roots of unity and showed that the real parts of their analytic limits yield symmetric
multiple zeta values, whereas their algebraic reductions yield finite
multiple zeta values, providing further evidence for the Kaneko--Zagier
conjecture.

\subsection{Function Field Counterpart}
We now turn to the function field setting. Let
$p$ be a prime number and let $q$ be a power of $p$. We let
$A=\FF_q[\theta]$ denote the polynomial ring over $\FF_q$ in the variable
$\theta$, and let $K=\FF_q(\theta)$ be its field of fractions. Let
$K_\infty=\FF_q(\!(1/\theta)\!)$ be the completion of $K$ at the infinite
place and $\CC_\infty$ the completion of an algebraic closure of
$K_\infty$. We also let $\ov{K}$ denote the algebraic closure of $K$ in
$\CC_\infty$.

Thakur introduced the function field analogue of multiple zeta values
$\zeta_A(\ww{s})$ in \cite{Tha04} for $\ww{s}\in\Ical$. Throughout this paper, we call them
the \emph{$\infty$-adic multiple zeta values}. Chang~\cite{Cha14} later introduced the \emph{Carlitz multiple polylogarithms} and expressed the $\infty$-adic multiple zeta values as $K$-linear combinations of Carlitz multiple polylogarithms evaluated at integral points, generalizing the work of~Anderson and Thakur \cite{AT90} for $\dep(\ww{s})=1$.
We will mainly consider their special values
$\Li_{\ww{s}}(\mathbf{1})$, where $\ww{s}\in\Ical$ and $\mathbf{1}$ denotes the tuple
$(1,\ldots,1)$ of the appropriate depth, as we have $\Li_{\ww{s}}(\mathbf{1})=\zeta_A(\ww{s})$ whenever $\ww{s}\in\Ical^{\rm{T}}$, the set of Thakur's indices (see \S\ref{sec:linear-relations}).
More recently, Mishiba~\cite{Mis26} introduced the corresponding
$\star$-inverse values $\zeta_A^{\ant}(\ww{s})$ and
$\Li_{\ww{s}}^\ant(\mathbf{1})$. These values are closely related to each other and we refer to \S\ref{sec:algebraic-setup} for the precise definitions.

We let $\Zcal_{\infty}$ and $\Zcal_{\infty,w}$ denote, respectively, the $K$-vector space spanned by all $\infty$-adic
multiple zeta values and the $K$-vector space spanned by those of weight
$w$, for each $w\geq 0$. Thakur~\cite{Tha10} proved that multiple zeta values satisfy the
\textit{$q$-shuffle relations}. In particular, $\Zcal_{\infty}$ forms a $K$-algebra
(see \S\ref{sec:Shuffle and stuffle} for details). Furthermore, for any $\ww{s}\in \Ical$, we have $\Li_{\ww{s}}(\mathbf{1}), \zeta_A^{\ant}(\ww{s}), \Li_{\ww{s}}^\ant(\mathbf{1})\in\Zcal_\infty$ (see, for instance, \S\ref{sec:linear-relations} and \cref{lem:infty-star-inverse-space}). Moreover, after passing to the quotient by the ideal $\zeta_A(q-1)\Zcal_\infty$, the values $\zeta_A(\ww{s})$ and $\Li_{\ww{s}}(\mathbf{1})$ satisfy the same algebraic relations as $\zeta_A^{\ant}(\ww{s})$ and $\Li_{\ww{s}}^\ant(\mathbf{1})$, respectively (see~\cite{Mis26,TsuCoaction26}).

In analogy with \eqref{eq:Acal}, we put
\[
\Acal_K
=
\left(\prod_v A/(v)\right)
\bigg/
\left(\bigoplus_v A/(v)\right),
\]
where $v$ runs through all monic irreducible polynomials of $A$.
There is a natural embedding $K\hookrightarrow\Acal_K$, and we regard $\Acal_K$ as a $K$-algebra. We recall that the \emph{finite multiple zeta values over $K$}, $\zeta_{\Acal_K}(\ww{s})\in\Acal_K$, were studied in
\cite{CM17,Shi18}. Chang and Mishiba introduced the
corresponding special values of \emph{finite Carlitz multiple
polylogarithms} $\Li_{\Acal_K,\ww{s}}(\mathbf{1})\in\Acal_K$ in \cite{CM17}.
We will also consider the corresponding $\star$-inverse values $\zeta^\ant_{\Acal_K}(\ww{s}),\Li^\ant_{\Acal_K,\ww{s}}(\mathbf{1})\in\Acal_K$ where the precise definitions are given in  \S\ref{sec:algebraic-setup}. These values may be viewed as finite analogues of their corresponding $\infty$-adic counterparts and share many analogous properties.

We write $\Zcal_{\Acal_K}$ and $\Zcal_{\Acal_K,w}$ for the $K$-vector spaces spanned by all finite multiple zeta values over $K$ and by those of weight $w$, respectively, for each $w\geq 0$. It follows from~\cite{CM17,Shi18} that finite multiple zeta values over $K$ also satisfy the $q$-shuffle relations. In particular, $\Zcal_{\Acal_K}$ forms a $K$-algebra (see \S\ref{sec:Shuffle and stuffle} for details). Furthermore, for any $\ww{s}\in\Ical$, the values
$\Li_{\Acal_K,\ww{s}}(\mathbf{1})$,
$\zeta^\ant_{\Acal_K}(\ww{s})$, and
$\Li^\ant_{\Acal_K,\ww{s}}(\mathbf{1})$
also belong to $\Zcal_{\Acal_K}$ (see \cref{cor:finite-realization-span,lem:finite-star-inverse-space}).

Motivated by the classical Kaneko--Zagier conjecture, the author proposed a function field analogue in \cite{Tsu26}.
To formulate it, we put
\[
\gfk_{\Acal_K}
=
\zeta_{\Acal_K}(q)
-
L_1\zeta_{\Acal_K}(1,q-1)
\in\Zcal_{\Acal_K},
\]
the element arising from Thakur's fundamental relation
\cite[Theorem~5]{Tha09}, where $L_1=\theta-\theta^q$. The corresponding conjecture is stated as follows.

\begin{conj}[{{\cite[Question~3.24 and Remark~3.25]{Tsu26}}}]
\label{conj:ffKZ}
There exists a $K$-algebra isomorphism
\[
\Phi_{\infty,\Acal_K}:
{\Zcal_{\infty}}
/{\zeta_A(q-1)\Zcal_{\infty}}
\xrightarrow{\ \sim\ }
{\Zcal_{\Acal_K}}/
{\gfk_{\Acal_K}\Zcal_{\Acal_K}}
\]
such that, for every $\ww{s}\in\Ical$,
\[
\Phi_{\infty,\Acal_K}
\left(
\zeta_A(\ww{s})
\bmod \zeta_A(q-1)\Zcal_{\infty}
\right)
=
\zeta_{\Acal_K}(\ww{s})
\bmod \gfk_{\Acal_K}\Zcal_{\Acal_K}.
\]
\end{conj}

This conjecture is motivated by the following observations.
In \cite{Tsu26}, the author introduced
\emph{finite multiple harmonic $u$-series} using the Carlitz module.
When evaluated at Carlitz torsion points, their analytic limits
recover ordinary $\infty$-adic multiple zeta values,
while their algebraic reductions recover finite multiple zeta values
over $K$ (see \cite[Theorems~3.18 and~3.19]{Tsu26}).
Unlike the classical result of Bachmann, Takeyama, and Tasaka
\cite{BTT18}, this construction does not involve
symmetric multiple zeta values, suggesting a direct comparison.
On the other hand, the choice of the quotient
by $\gfk_{\Acal_K}$ is motivated by Shi's computations
of relations among finite multiple zeta values and her
conjectural dimension formula \cite{Shi18}
(see also \cref{lem:finite-CCM-relations}).

\subsection{Main Results}
Our main result in this paper establishes a connection between
$\infty$-adic and finite multiple zeta values. More precisely, we construct
a well-defined surjective comparison map from the algebra of $\infty$-adic multiple
zeta values to a quotient of the algebra of finite multiple zeta values.
This map factors through the quotient by the ideal $\zeta_A(q-1)\Zcal_\infty$, providing
the well-definedness part for \cref{conj:ffKZ}.

\begin{thm}[Restated as
\cref{thm:finite-comparison-map,thm:finite-star-inverse-comparison}]
\label{thm:main-finite-comparison}
There exists a well-defined surjective $K$-algebra homomorphism
\[
\widetilde{\Phi}_{\infty,\Acal_K}:
\Zcal_{\infty}
\longtwoheadrightarrow
{\Zcal_{\Acal_K}}
/{\gfk_{\Acal_K}\Zcal_{\Acal_K}}
\]
such that, for every $\ww{s}\in\Ical$,
\[
\widetilde{\Phi}_{\infty,\Acal_K}
\left(
\zeta_A(\ww{s})
\right)
=
\zeta_{\Acal_K}(\ww{s})
\bmod
\gfk_{\Acal_K}\Zcal_{\Acal_K}.
\]
Moreover, the map is compatible with the corresponding special values
of Carlitz multiple polylogarithms and with the $\star$-inverse values:
\[
\begin{aligned}
\widetilde{\Phi}_{\infty,\Acal_K}
\left(
\Li_{\ww{s}}(\mathbf{1})
\right)
&=
\Li_{\Acal_K,\ww{s}}(\mathbf{1})
\bmod
\gfk_{\Acal_K}\Zcal_{\Acal_K},\\
\widetilde{\Phi}_{\infty,\Acal_K}
\left(
\zeta_A^\ant(\ww{s})
\right)
&=
\zeta_{\Acal_K}^\ant(\ww{s})
\bmod
\gfk_{\Acal_K}\Zcal_{\Acal_K},\\
\widetilde{\Phi}_{\infty,\Acal_K}
\left(
\Li_{\ww{s}}^\ant(\mathbf{1})
\right)
&=
\Li_{\Acal_K,\ww{s}}^\ant(\mathbf{1})
\bmod
\gfk_{\Acal_K}\Zcal_{\Acal_K}.
\end{aligned}
\]
Furthermore,
\[
\zeta_A(q-1)\Zcal_{\infty}
\subseteq
\ker\widetilde{\Phi}_{\infty,\Acal_K},
\]
and hence $\widetilde{\Phi}_{\infty,\Acal_K}$ induces a surjective
$K$-algebra homomorphism
\[
\Phi_{\infty,\Acal_K}:
{\Zcal_{\infty}}
/{\zeta_A(q-1)\Zcal_{\infty}}
\longtwoheadrightarrow
{\Zcal_{\Acal_K}}
/{\gfk_{\Acal_K}\Zcal_{\Acal_K}}.
\]
\end{thm}

The key idea in the proof of \cref{thm:main-finite-comparison} is to use the explicit generators in \cite[\S3]{CCM23} for the kernel of the $\infty$-adic realization (see \S\ref{sec:linear-relations} for details) and to show that their finite realizations vanish modulo $\gfk_{\Acal_K}\Zcal_{\Acal_K}$ (see~\cref{lem:finite-CCM-relations}). This establishes the well-definedness
of the map. To prove its compatibility with the $\star$-inverse values, we use
the formalism of evaluation maps developed in~\cite[\S3.1]{TsuCoaction26}. As an application of \cref{thm:main-finite-comparison}, we also obtain an upper
bound for the dimension of $\Zcal_{\Acal_K,w}$ (see
\cref{thm:finite-dimension-upper-bound}).

The rest of the paper is organized as follows. In \S\ref{sec:preliminary}, we recall the necessary notation and background.
In \S\ref{sec:formalism}, we recall the formalism of evaluation maps
that will be used in the sequel.
In \S\ref{sec:finite-comparison}, we construct the comparison map from
$\infty$-adic to finite multiple zeta values and prove its compatibility
with the two realizations and their $\star$-inverse counterparts.
Finally, in \S\ref{sec:dimension}, we compare the dimension problems in
the $\infty$-adic, $v$-adic, and finite settings, establish the finite
dimension upper bound, and formulate a unified comparison conjecture.
\section{Preliminaries}\label{sec:preliminary}

\subsection{Algebraic Setup}\label{sec:algebraic-setup}
We now fix the algebraic notation used throughout the paper. Recall that
\[
\Ical=\{\varnothing\}\cup\bigcup_{r\geq1}\NN^r
\]
denotes the set of indices. We let
\[
\Hfk=\Span_K\{[\ww{s}]:\ww{s}\in\Ical\}
\]
denote the free $K$-vector space generated by the
set of indices $\Ical$. For each $w\geq0$, we also define
\[
\Hfk_w
=
\Span_K
\{
[\ww{s}]
:
\ww{s}\in\Ical,\ \wt(\ww{s})=w
\}.
\]
We put $1=[\varnothing]$.
Then the concatenation map
\[
[-,-]:\Hfk\times\Hfk\rightarrow\Hfk
\]
is the unique $K$-bilinear map determined by
\[
[\ww{s},\ww{n}]
=
[s_1,\ldots,s_r,n_1,\ldots,n_\ell],
\qquad
[\varnothing,\ww{s}]
=
[\ww{s},\varnothing]
=
[\ww{s}],
\qquad
[\varnothing,\varnothing]
=
[\varnothing],
\]
for all non-empty indices
$\ww{s}=(s_1,\ldots,s_r),\ww{n}=(n_1,\ldots,n_\ell)\in\Ical$.
Iterated concatenations are understood $K$-multilinearly and are denoted by
\[
[-,\ldots,-]\colon
\Hfk\times\cdots\times\Hfk
\longrightarrow
\Hfk.
\]
We identify an index with its corresponding basis element when no
confusion can arise.

We define the deconcatenation map 
\[
\Delta:\Hfk\longrightarrow\Hfk\otimes_K\Hfk,
\]
which is the unique $K$-linear map such that 
\begin{equation}\label{eq:Delta}
    \Delta([\ww{s}])
=
\sum_{\ww{s}=[\ww{s}_1,\ww{s}_2]}
[\ww{s}_1]
\otimes
[\ww{s}_2]
\end{equation}
for all indices $\ww{s}\in\Ical$.

We next define the realization maps used throughout the paper. For $d\geq0$ and $s\in\NN$, set
\[
S_d(s)
=
\sum_{\substack{a\in A \text{ monic}\\\deg a=d}}
\frac{1}{a^s},
\qquad
L_0=1,
\qquad
L_d=\prod_{i=1}^d(\theta-\theta^{q^i})
\quad(d\geq1).
\]
For $\bullet\in\{\zeta,\Li\}$, put
\[
\mathscr{S}_d^\zeta(s)=S_d(s),
\qquad
\mathscr{S}_d^\Li(s)=\frac{1}{L_d^s}.
\]

\begin{df}\label{df:infty-adic-realizations}
Let $\bullet\in\{\zeta,\Li\}$. We define the $K$-linear maps $\mathscr{L}_\infty^\bullet,\mathscr{L}_\infty^{\bullet,\ant}:\Hfk\to K_\infty$ as follows. For any non-empty index $\ww{s}=(s_1,\ldots,s_r)\in\Ical$, define
\begin{align*}
\mathscr{L}_\infty^\bullet([\ww{s}])
&=
\sum_{d_1>\cdots>d_r\geq0}
\prod_{i=1}^r\mathscr{S}_{d_i}^\bullet(s_i)\in K_\infty,\\
\mathscr{L}_\infty^{\bullet,\ant}([\ww{s}])
&=
(-1)^r
\sum_{0\leq d_1\leq\cdots\leq d_r}
\prod_{i=1}^r\mathscr{S}_{d_i}^\bullet(s_i)\in K_\infty.
\end{align*}
Both maps take $1=[\varnothing]$ to $1$ and are extended $K$-linearly
to $\Hfk$.
\end{df}

For every $\ww{s}\in\Ical$, we have
\[
\mathscr{L}_\infty^\zeta([\ww{s}])=\zeta_A(\ww{s}),
\qquad
\mathscr{L}_\infty^{\zeta,\ant}([\ww{s}])=\zeta_A^\ant(\ww{s}),
\]
and
\[
\mathscr{L}_\infty^\Li([\ww{s}])=\Li_{\ww{s}}(\mathbf{1}),
\qquad
\mathscr{L}_\infty^{\Li,\ant}([\ww{s}])
=
\Li_{\ww{s}}^\ant(\mathbf{1}).
\]
Moreover, $\Zcal_{\infty,w}$ is also spanned by
$\mathscr{L}_\infty^\Li([\ww{s}])=\Li_{\ww{s}}(\mathbf{1})$ with
$\wt(\ww{s})=w$ (see \cite[Corollary~3.7]{CCM23}).

\begin{df}\label{df:finite-realizations}
Let $\bullet\in\{\zeta,\Li\}$.  We define the $K$-linear maps $\mathscr{L}_{\Acal_K}^\bullet,\mathscr{L}_{\Acal_K}^{\bullet,\ant}:\Hfk\to \Acal_K$ as follows. For any non-empty index $\ww{s}=(s_1,\ldots,s_r)\in\Ical$, define
\[
\mathscr{L}_{\Acal_K}^\bullet([\ww{s}])
=
\left(
\sum_{\deg v>d_1>\cdots>d_r\geq0}
\prod_{i=1}^r\mathscr{S}_{d_i}^\bullet(s_i)
\bmod v
\right)_v
\in\Acal_K
\]
and
\[
\mathscr{L}_{\Acal_K}^{\bullet,\ant}([\ww{s}])
=
\left(
(-1)^r
\sum_{0\leq d_1\leq\cdots\leq d_r<\deg v}
\prod_{i=1}^r\mathscr{S}_{d_i}^\bullet(s_i)
\bmod v
\right)_v
\in\Acal_K.
\]
Both maps take $1=[\varnothing]$ to $1$ and are extended $K$-linearly
to $\Hfk$.
\end{df}

For every $\ww{s}\in\Ical$, we have
\[
\mathscr{L}_{\Acal_K}^\zeta([\ww{s}])
=
\zeta_{\Acal_K}(\ww{s}),
\qquad
\mathscr{L}_{\Acal_K}^{\zeta,\ant}([\ww{s}])
=
\zeta_{\Acal_K}^{\ant}(\ww{s}),
\]
and
\[
\mathscr{L}_{\Acal_K}^\Li([\ww{s}])
=
\Li_{\Acal_K,\ww{s}}(\mathbf{1}),
\qquad
\mathscr{L}_{\Acal_K}^{\Li,\ant}([\ww{s}])
=
\Li_{\Acal_K,\ww{s}}^\ant(\mathbf{1}).
\]
By Carlitz's formula
\begin{equation}\label{eq:Carlitz-formula}
    \mathscr{S}_d^\zeta(s)=\mathscr{S}_d^\Li(s)
\qquad (1\leq s\leq q)
\end{equation}
(see \cite[Theorem~5.9.1]{Tha04}), we see that the element
\[
\gfk_{\Acal_K}
:=
\mathscr{L}_{\Acal_K}^\bullet([q])
-
L_1\mathscr{L}_{\Acal_K}^\bullet([1,q-1])
\]
is independent of the choice of
$\bullet\in\{\zeta,\Li\}$.

\subsection{Shuffle and Stuffle Products}\label{sec:Shuffle and stuffle}
We next recall the two products on $\Hfk$ corresponding to the
$\zeta$- and $\Li$-realizations. In \cite{Tha10}, Thakur proved the
existence of the $q$-shuffle relations, showing that the product of two
$\infty$-adic multiple zeta values can be expressed as an
$\FF_p$-linear combination of $\infty$-adic multiple zeta values of the
same total weight. In the case of depth one, Chen obtained the following explicit
formula in \cite{Che15}:
\[
\begin{aligned}
\zeta_A(s)\zeta_A(n)
={}&
\zeta_A(s,n)
+\zeta_A(n,s)
+\zeta_A(s+n)
+
\sum_{{j=1}}^{s+n-1}
\Delta_{s,n}^{[j]}
\zeta_A(s+n-j,j)
\end{aligned}
\]
for $s,n\in\NN$. Here, we write
\[
\Delta_{s,n}^{[j]}
=
\begin{cases}
\displaystyle
(-1)^{s-1}\binom{j-1}{s-1}
+
(-1)^{n-1}\binom{j-1}{n-1},
& q-1\mid j,\\[6pt]
0,
& \text{otherwise},
\end{cases}
\]
for $s,n,j\in\NN$.
Based on Chen's formula, Yamamoto observed a recursive construction of
the $q$-shuffle product (see \cite[\S5.2]{Tha17}), which was later formulated and proved by Shi
(see \cite[Definition~3.1.3 and Theorem~3.1.4]{Shi18}).
On the other hand, the corresponding product on the special
values of Carlitz multiple polylogarithms $\Li_{\ww{s}}(\mathbf{1})$ is the usual stuffle product
\cite{Cha14}, which is given in \cref{df:bullet-stuffle} for $\bullet=\Li$. 

For non-empty $\ww{s}=(s_1,\ldots,s_r)\in\Ical$, we put
\[
\ww{s}^{-}=(s_2,\ldots,s_r),
\qquad
\varnothing^-=\varnothing.
\]
Then we recall the following definition from \cite[\S2.3]{CCM23}.

\begin{df}\label{df:bullet-stuffle}
Let $\bullet\in\{\zeta,\Li\}$.
We define two $K$-bilinear products on $\Hfk$, the $q$-shuffle product
$\ast^\zeta$ and the stuffle product $\ast^\Li$, by
\[
1\ast^\bullet P
=
P\ast^\bullet1
=
P
\]
and, for non-empty indices
$\ww{s}=(s_1,\ldots,s_r)$ and
$\ww{n}=(n_1,\ldots,n_\ell)$, by
\begin{equation}\label{eq:bullet-stuffle}
\begin{aligned}
\ww{s}\ast^\bullet\ww{n}
={}&
[s_1,\ww{s}^-\ast^\bullet\ww{n}]
+[n_1,\ww{s}\ast^\bullet\ww{n}^-]\\
&+
[s_1+n_1,\ww{s}^-\ast^\bullet\ww{n}^-]
+D_{\ww{s}}^\bullet(\ww{n}).
\end{aligned}
\end{equation}
Here we put $D_{\ww{s}}^\Li(\ww{n})=0$, and
\begin{equation}\label{eq:D-zeta-definition}
D_{\ww{s}}^\zeta(\ww{n})
=
\sum_{j=1}^{s_1+n_1-1}
\Delta_{s_1,n_1}^{[j]}
[s_1+n_1-j,
[j]\ast^\zeta(\ww{s}^-\ast^\zeta\ww{n}^-)].
\end{equation}
For every non-empty $\ww{s}\in\Ical$, we set
$D_{\ww{s}}^\bullet(1)=0$ and extend
$D_{\ww{s}}^\bullet(-)$ to $\Hfk$ $K$-linearly.
For convenience, we write
\[
D_c^\bullet(-)=D_{(c)}^\bullet(-),
\qquad
D_{\ww{s}}(-)=D_{\ww{s}}^\zeta(-).
\]
\end{df}

The products $\ast^\zeta$ and $\ast^\Li$ on $\Hfk$ are compatible with the
corresponding $\infty$-adic realizations as follows.

\begin{thm}[{\cite{Cha14,Shi18}; see also
\cite[Proposition~2.7]{CCM23}}]
\label{thm:strict-bullet-character}
Let $\bullet\in\{\zeta,\Li\}$.
For any $P,Q\in\Hfk$, we have
\[
\mathscr{L}_\infty^\bullet(P\ast^\bullet Q)
=
\mathscr{L}_\infty^\bullet(P)
\mathscr{L}_\infty^\bullet(Q).
\]
\end{thm}

The finite realizations satisfy the same product formulas, stated as follows.
\begin{thm}[{\cite[\S3.1]{CM17} and \cite[\S4]{Shi18}}]\label{thm:finite-character}
   Let $\bullet\in\{\zeta,\Li\}$.
For any $P,Q\in\Hfk$, we have
\[
\mathscr{L}_{\Acal_K}^\bullet(P\ast^\bullet Q)
=
\mathscr{L}_{\Acal_K}^\bullet(P)
\mathscr{L}_{\Acal_K}^\bullet(Q).
\]
\end{thm}

\subsection{Linear Relations}\label{sec:linear-relations}
We next recall the result of Thakur's basis conjecture for $\infty$-adic multiple zeta values
that will be used in the proof of \cref{thm:main-finite-comparison}. In \cite[Conjecture~7.1]{Tod18}, Todd first conjectured a
formula for $\dim_K\Zcal_{\infty,w}$ for each weight. Later, Thakur predicted an
explicit basis indexed by the following set
(see \cite[Conjecture~8.2]{Tha17}):
\[
\Ical_w^{\mathrm{T}}
=
\begin{cases}
\{\varnothing\},
& w=0,\\[1mm]
\left\{
\ww{s}=(s_1,\ldots,s_r)\in\Ical :
s_i\leq q\ (1\leq i<r),
\ s_r<q,\ \wt(\ww{s})=w
\right\},
& w\geq1.
\end{cases}
\]
Let
\[
\Ical^{\mathrm{T}}
=
\bigcup_{w\geq0}\Ical_w^{\mathrm{T}}
\]
be the set of Thakur's indices.
For $\ww{s}\in\Ical^{\mathrm{T}}$, one has (see \cite[(2.7)]{CCM23})
\[
\zeta_A(\ww{s})
=
\mathscr{L}_\infty^\zeta([\ww{s}])
=
\mathscr{L}_\infty^\Li([\ww{s}])
=
\Li_{\ww{s}}(\mathbf{1}).
\]
Ngo Dac proved that these values span all
multiple zeta values over $K$
(see \cite[Theorem~A]{ND21}), and the resulting basis conjecture was later proved independently by Chang--Chen--Mishiba \cite{CCM23} and by Im--Kim--Le--Ngo Dac--Pham \cite{IKLNP24}.

\begin{thm}[{\cite[Theorem~1.5]{CCM23},
\cite[Theorem~B]{IKLNP24}}]
\label{thm:basis_for_infty}
Let $\bullet\in\{\zeta,\Li\}$.
Then for each $w\geq0$,
\[
\{
\mathscr{L}_\infty^\bullet([\ww{s}])
:
\ww{s}\in\Ical_w^{\mathrm{T}}
\}
\]
forms a $K$-basis of $\Zcal_{\infty,w}$.
\end{thm}

To prove the result, we also need the explicit generators for
the kernels of the two $\infty$-adic realizations, which were
developed in the framework of \cite{CCM23}. We recall the formulation
used here.

\begin{df}[{\cite[\S3]{CCM23}}]
For $c\in\NN$, $\ww{s}\in\Ical$, and
$\bullet\in\{\zeta,\Li\}$, define the $K$-linear map
\[
\alpha_{c;\ww{s}}^\bullet:
\Hfk
\longrightarrow
\Hfk,
\qquad
\alpha_{c;\ww{s}}^\bullet(P)
=
[c,\ww{s}\ast^\bullet P].
\]
For $m\geq1$, put
\[
\alpha_{c;\ww{s}}^{\bullet,m}
=
\underbrace{
\alpha_{c;\ww{s}}^\bullet
\circ\cdots\circ
\alpha_{c;\ww{s}}^\bullet
}_{m\text{ times}},
\qquad
\alpha_{c;\ww{s}}^{\bullet,0}
=
\id_{\Hfk}.
\]
We also define the $K$-bilinear map
\[
\boxplus:
\Hfk\times\Hfk
\longrightarrow
\Hfk
\]
by
\[
1\boxplus P
=
P\boxplus1
=
0\qquad\text{and}\qquad
\ww{s}\boxplus\ww{n}
=
[s_1,\ldots,s_{r-1},s_r+n_1,n_2,\ldots,n_\ell]
\]
for non-empty indices
$\ww{s}=(s_1,\ldots,s_r)$ and
$\ww{n}=(n_1,\ldots,n_\ell)\in\Ical$.
Here $0$ denotes the zero vector of $\Hfk$, rather than the empty word
$1=[\varnothing]$.
\end{df}

For non-empty $\ww{s}=(s_1,\ldots,s_r)\in\Ical$, we let
\[
\ww{s}^{+}=(s_1,\ldots,s_{r-1}),
\qquad
\varnothing^+=\varnothing.
\]
Also, for $c\in\NN$ and $m\geq1$, we write
$\{c\}^m=(c,\ldots,c)$ for the index consisting of $m$ copies of $c$
and $\{c\}^0=\varnothing$.
Following \cite{Mis26}, we recall the following
explicit formulation of the kernel relations
from \cite{CCM23}.
\begin{df}[{\cite[\S2]{Mis26}; see also \cite[\S3]{CCM23}}]
\label{df:Mishiba-relations}
Let $\bullet\in\{\zeta,\Li\}$.
Put $\varepsilon_\zeta=1$ and $\varepsilon_\Li=0$.
For $\ww{s},\ww{n}\in\Ical$ and $m\geq1$, define
\[
\begin{aligned}
\sA^\bullet(\ww{s};m;\ww{n})
={}&
[\ww{s},\{q\}^m,\ww{n}]
+
[\ww{s},\{q\}^m\boxplus\ww{n}]
+
\varepsilon_\bullet
[\ww{s},\{q\}^{m-1},D_q(\ww{n})]\\
&-
L_1^m
[\ww{s},\alpha_{1;(q-1)}^{\bullet,m}(\ww{n})]
-
L_1^m
[\ww{s}]
\boxplus
\alpha_{1;(q-1)}^{\bullet,m}(\ww{n})\\
&-
\varepsilon_\bullet L_1^m
[\ww{s}^+,
D_{s_r}
(\alpha_{1;(q-1)}^{\bullet,m}(\ww{n}))],
\end{aligned}
\]
if $\ww{s}=(s_1,\ldots,s_r)$ is non-empty, and
\[
\begin{aligned}
\sA^\bullet(\varnothing;m;\ww{n})
={}&
[\{q\}^m,\ww{n}]
+
[\{q\}^m\boxplus\ww{n}]
+
\varepsilon_\bullet
[\{q\}^{m-1},D_q(\ww{n})]
-
L_1^m
\alpha_{1;(q-1)}^{\bullet,m}(\ww{n}).
\end{aligned}
\]
We also set
\[
\mathscr{R}^\bullet
=
\Span_K
\left\{
\sA^\bullet(\ww{s};m;\ww{n})
:
\ww{s},\ww{n}\in\Ical,\ m\geq1
\right\}
\subseteq
\Hfk
\]
and
\[
\mathscr{R}_{w}^\bullet
=
\Span_K
\left\{
\sA^\bullet(\ww{s};m;\ww{n})
:
\begin{array}{c}
\ww{s},\ww{n}\in\Ical,\ m\geq1,\\
\wt(\ww{s})+mq+\wt(\ww{n})=w
\end{array}
\right\}
\subseteq
\Hfk_w.
\]
\end{df}

The kernels of the corresponding $\infty$-adic realization maps are generated by the elements $\sA^\bullet(\ww{s};m;\ww{n})$, as follows.

\begin{prop}[{\cite[Proposition~2.3.2]{Mis26}; see also \cite[\S3]{CCM23}}]
\label{prop:Mishiba-kernel}
For every $\bullet\in\{\zeta,\Li\}$,
$\ww{s},\ww{n}\in\Ical$, and $m\geq1$, we have
\[
\mathscr{L}_\infty^\bullet
\left(
\sA^\bullet(\ww{s};m;\ww{n})
\right)
=
0.
\]
Moreover, for all $w\geq0$,
\[
\ker
\left(
\mathscr{L}_\infty^\bullet|_{\Hfk_w}:
\Hfk_w
\longrightarrow
\Zcal_{\infty}
\right)
=
\mathscr{R}_{w}^\bullet,
\qquad
\ker
\left(
\mathscr{L}_\infty^\bullet:
\Hfk
\longrightarrow
\Zcal_{\infty}
\right)
=
\mathscr{R}^\bullet.
\]
\end{prop}

\subsection{Binary Relations}
We next recall the notion of binary relations and the operators
$\mathscr{B}$ and $\mathscr{C}$, introduced by Todd
\cite[\S3]{Tod18}, as well as the $\mathscr{BC}$ operator introduced by
Ngo Dac \cite{ND21}. We use the
simultaneous $\zeta$- and $\Li$-formulation of these constructions given in
\cite[Appendix~A.2--A.3]{CCM23} as follows:

For $D\in\ZZ$, we define $K$-linear maps
\[
\mathscr{L}_D^\bullet,\,
\mathscr{L}_{<D}^\bullet:
\Hfk\longrightarrow K
\]
as follows. We define
\[
\mathscr{L}_D^\bullet([\varnothing])
=
\begin{cases}
1,& D=-1,\\
0,& D\neq -1,
\end{cases}
\qquad
\mathscr{L}_{<D}^\bullet([\varnothing])
=
\begin{cases}
1,& D\geq0,\\
0,& D<0.
\end{cases}
\] For a non-empty index
$\ww{s}=(s_1,\ldots,s_r)\in\Ical$, we set
\[
\mathscr{L}_D^\bullet([\ww{s}])
=
\sum_{D=d_1>d_2>\cdots>d_r\geq0}
\prod_{i=1}^r
\mathscr{S}_{d_i}^\bullet(s_i)
\]
and
\[
\mathscr{L}_{<D}^\bullet([\ww{s}])
=
\sum_{D>d_1>\cdots>d_r\geq0}
\prod_{i=1}^r
\mathscr{S}_{d_i}^\bullet(s_i).
\]
Then for every non-empty
$\ww{s}=(s_1,\ww{s}^-)\in\Ical$,
\[
\mathscr{L}_D^\bullet([\ww{s}])
=
\mathscr{S}_D^\bullet(s_1)
\mathscr{L}_{<D}^\bullet([\ww{s}^-])\qquad(D\geq 0),
\]
and
\[
\mathscr{L}_{<D}^\bullet([\ww{s}])
=
\sum_{0\leq d<D}
\mathscr{L}_d^\bullet([\ww{s}])\qquad (D\in\ZZ).
\]
In particular, for every $P\in\Hfk$, we have
\[
\mathscr{L}_{\Acal_K}^\bullet(P)
=
\left(\mathscr{L}_{<\deg v}^\bullet(P)\bmod v\right)_v\in\Acal_K.
\]

Let $\bullet\in\{\zeta,\Li\}$ and \[\Hfk^+=\Span_K\{[\ww{s}]:\varnothing\neq\ww{s}\in\Ical\}.\] A pair $(P,Q)\in\Hfk^+\times\Hfk^+$ is called a
\emph{binary relation} for the $\bullet$-realization if
\[
\mathscr{L}_d^\bullet(P)
+
\mathscr{L}_{d+1}^\bullet(Q)
=
0
\]
for all $d\in\ZZ$. Then Thakur's fundamental binary relation is
\begin{equation}\label{eq:R_1}
    R_1
=
\left([q],-L_1[1,q-1]\right),
\end{equation}
corresponding to
\[
\mathscr{L}_d^\bullet([q])
-
L_1\mathscr{L}_{d+1}^\bullet([1,q-1])
=
0
\]
for all $d\in\ZZ$ and both $\bullet\in\{\zeta,\Li\}$ (see \cite[Theorem~5]{Tha09} and
\cite[Appendix~A.2]{CCM23}).

We also recall the operators on binary relations following
\cite[Appendix~A.3]{CCM23}. Let
$\ww{s}=(s_1,\ldots,s_r)\in\Ical$ be non-empty, and write
\[
P=\sum_{\ww{n}\neq\varnothing}a_{\ww{n}}[\ww{n}],
\qquad
Q=\sum_{\ww{n}\neq\varnothing}b_{\ww{n}}[\ww{n}].
\]
Put
$D_{c,Q}^\bullet
=
\sum_{\ww{n}\neq\varnothing}
b_{\ww{n}}D_c^\bullet(\ww{n})$.
Then we define
\[
\mathscr{B}_{\ww{s}}^\bullet(P,Q)
=
\left(
[\ww{s},P]
+
[\ww{s},Q]
+
[\ww{s}]\boxplus Q
+
[\ww{s}^+,D_{s_r,Q}^\bullet],
\,0
\right),
\]
and
\[
\begin{aligned}
\mathscr{C}_{\ww{s}}^\bullet(P,Q)
=
\sum_{\substack{
\ww{n}=(n_1,\ww{n}^-)\in\Ical\\
\ww{n}\neq\varnothing}}
\bigl(
&
a_{\ww{n}}
[n_1+s_1,\ww{n}^-\ast^\bullet\ww{s}^-]
+
a_{\ww{n}}
[n_1,\ww{n}^-\ast^\bullet\ww{s}]+
a_{\ww{n}}D_{\ww{n}}^\bullet(\ww{s}),
\,
b_{\ww{n}}
[n_1,\ww{n}^-\ast^\bullet\ww{s}]
\bigr).
\end{aligned}
\]
For $m\geq0$, we further put
\[
\mathscr{BC}_q^{\bullet,m}(P,Q)
=
\left(
[\{q\}^m,P],
\,
L_1^m
\alpha_{1;(q-1)}^{\bullet,m}(Q)
\right).
\]
We set
$\mathscr{B}_\varnothing^\bullet
=
\mathscr{C}_\varnothing^\bullet
=
\id$.
Then these operators preserve binary relations by
\cite[Proposition~A.4]{CCM23}. Finally, let
\[
\beta:\Hfk^+\times\Hfk^+\longrightarrow\Hfk^+,
\qquad
\beta(P,Q)=P+Q.
\]
Then one notices that, for every
$\ww{s},\ww{n}\in\Ical$ and $m\geq1$,
\[
\sA^\bullet(\ww{s};m;\ww{n})
=
\beta\left(
\mathscr{B}_{\ww{s}}^\bullet
\left(
\mathscr{BC}_q^{\bullet,m-1}
\left(
\mathscr{C}_{\ww{n}}^\bullet(R_1)
\right)
\right)
\right).
\]

\section{The Finite Comparison Map}\label{sec:comparison}
In this section, we construct the finite comparison map
using the explicit description of the kernels of the
$\infty$-adic realization maps recalled in the previous
section. In \S\ref{sec:formalism}, we recall the formalism of evaluation maps developed in \cite[\S3.1]{TsuCoaction26}. In \S\ref{sec:finite-comparison}, we construct the desired map and establish its compatibility with the corresponding $\star$-inverse values by using the formalism in \S\ref{sec:formalism}.
\subsection{Formalism of Evaluation Maps}\label{sec:formalism}
Throughout this subsection, we let $\bullet\in\{\zeta,\Li\}$, $B$ be a commutative unital $K$-algebra, and let
\[
\ev,\ev^\ant:\Hfk\longrightarrow B
\]
be $K$-linear maps. We consider the following conditions:
\begin{enumerate}[label=\textup{(C\arabic*)},
ref=\textup{(C\arabic*)}]
\item\label{cond:unital}
$\ev(1)=\ev^\ant(1)=1$.

\item\label{cond:convolution}
For every non-empty index $\ww{s}\in\Ical$,
\[
m_B(\ev\otimes\ev^\ant)\Delta([\ww{s}])
=
0
=
m_B(\ev^\ant\otimes\ev)\Delta([\ww{s}]),
\]
where $m_B$ denotes the multiplication map of $B$ and $\Delta$ is defined in \eqref{eq:Delta}.
\end{enumerate}
We also consider the following conditions, which depend on the choice of $\bullet\in\{\zeta,\Li\}$.
\begin{enumerate}[label=\ensuremath{(\mathrm{C}\arabic*)^\bullet},
ref=\ensuremath{(\mathrm{C}\arabic*)^\bullet},start=3]
\item\label{cond:vanishing}
We have $\ev^\ant([j])=0$ whenever $q-1\mid j$
if $\bullet=\zeta$, and $\ev^\ant([q-1])=0$
if $\bullet=\Li$.

\item\label{cond:ev-product}
For all $P,Q\in\Hfk$,
\[
\ev(P\ast^\bullet Q)=\ev(P)\ev(Q).
\]

\item\label{cond:ant-product}
For all $P,Q\in\Hfk$,
\[
\ev^\ant(P\ast^\bullet Q)=\ev^\ant(P)\ev^\ant(Q).
\]
\end{enumerate}

We will use the following three lemmas, which play key roles in compatibility results in the proofs of \cref{lem:finite-star-inverse-space,thm:finite-star-inverse-comparison}.

\begin{lem}[{\cite[Lemma 3.3]{TsuCoaction26}}]\label{lem:degree-convolution}
For every $s\in\NN$ and $d\geq0$, let
$\mathscr{S}_d(s)\in B$, and let
$D\in\NN\cup\{\infty\}$. Assume that the sums below are well-defined in
$B$. Define $K$-linear maps $\ev,\ev^\ant:\Hfk\to B$ by
$\ev(1)=\ev^\ant(1)=1$ and, for every non-empty index
$\ww{s}=(s_1,\ldots,s_r)$, by
\begin{align*}
\ev([\ww{s}])
&=
\sum_{D>d_1>\cdots>d_r\geq0}
\prod_{k=1}^r\mathscr{S}_{d_k}(s_k),\\
\ev^\ant([\ww{s}])
&=
(-1)^r
\sum_{0\leq d_1\leq\cdots\leq d_r<D}
\prod_{k=1}^r\mathscr{S}_{d_k}(s_k).
\end{align*}
Then Conditions \ref{cond:unital} and \ref{cond:convolution} hold.
\end{lem}

\begin{lem}[{\cite[Lemma 3.2]{TsuCoaction26}}]\label{lem:bullet-stuffle-transfer}
Let $\bullet\in\{\zeta,\Li\}$. Assume Conditions
\ref{cond:unital} and  \ref{cond:convolution}. Then the following hold:
\begin{enumerate}
    \item For $\bullet =\Li$, if Condition \ref{cond:ev-product} holds, then Condition \ref{cond:ant-product} holds.
    \item For $\bullet =\zeta$, if Conditions \ref{cond:vanishing} and \ref{cond:ev-product} hold, then Condition \ref{cond:ant-product} holds.
\end{enumerate}
\end{lem}

\begin{lem}[{\cite[Lemma 3.4]{TsuCoaction26}}]\label{lem:image-comparison}
Let $\bullet\in\{\zeta,\Li\}$. Assume Conditions \ref{cond:unital} and
\ref{cond:convolution}. Then the following hold:
\begin{enumerate}[label=\textup{(\roman*)}]
\item If Condition \ref{cond:ev-product} holds, then
\[
\ev^\ant(\Hfk)\subseteq\ev(\Hfk).
\]

\item If Condition \ref{cond:ant-product} holds, then
\[
\ev(\Hfk)\subseteq\ev^\ant(\Hfk).
\]
\end{enumerate}
In particular, if both Conditions \ref{cond:ev-product}
and \ref{cond:ant-product} hold, then
\[
\ev(\Hfk)=\ev^\ant(\Hfk).
\]
\end{lem}

\subsection{The Finite Comparison Map}\label{sec:finite-comparison}
We now turn to the finite setting. For
$\bullet\in\{\zeta,\Li\}$, let $\Zcal_{\Acal_K}^\bullet
\subseteq
\Acal_K$
be the $K$-algebra generated by
$\mathscr{L}_{\Acal_K}^\bullet([\ww{s}])$ with
$\ww{s}\in\Ical$. For $w\geq0$, put
\[
\Zcal_{\Acal_K,w}^\bullet
=
\Span_K
\left\{
\mathscr{L}_{\Acal_K}^\bullet([\ww{s}])
:
\ww{s}\in\Ical,\ \wt(\ww{s})=w
\right\},
\]
and set
$\Zcal_{\Acal_K,w}^\bullet=0$ for $w<0$.
We also let
\[
\Zcal_{\Acal_K,w}^{\bullet,\mathrm{T}}
=
\Span_K
\{
\mathscr{L}_{\Acal_K}^\bullet([\ww{s}])
:
\ww{s}\in\Ical_w^{\mathrm{T}}
\}
\]
be the $K$-vector space of corresponding values at the set of Thakur's indices of weight $w$. Note that
\[
\Zcal_{\Acal_K}^\bullet
=
\sum_{w\geq0}\Zcal_{\Acal_K,w}^\bullet,
\]
and by \cref{thm:finite-character}, we have
\[
\Zcal_{\Acal_K,w}^\bullet
\Zcal_{\Acal_K,w'}^\bullet
\subseteq
\Zcal_{\Acal_K,w+w'}^\bullet.
\]

We first determine the images of the generators
of the kernels of the $\infty$-adic realization
maps under the corresponding finite realizations.

\begin{lem}\label{lem:finite-CCM-relations}
Let $\bullet\in\{\zeta,\Li\}$,
$\ww{s},\ww{n}\in\Ical$, and $m\geq1$. Then
\[
\mathscr{L}_{\Acal_K}^\bullet
\left(
\sA^\bullet(\ww{s};m;\ww{n})
\right)
=
\begin{cases}
0,
&\ww{s}\neq\varnothing,\\[2mm]
\displaystyle
\gfk_{\Acal_K}
\mathscr{L}_{\Acal_K}^\bullet
\left(
L_1^{m-1}
\alpha_{1;(q-1)}^{\bullet,m-1}(\ww{n})
\right),
&\ww{s}=\varnothing.
\end{cases}
\]
Consequently, if
$P\in\ker(\mathscr{L}_\infty^\bullet)$ is homogeneous of weight $w$, then
\[
\mathscr{L}_{\Acal_K}^\bullet(P)
\in
\gfk_{\Acal_K}
\Zcal_{\Acal_K,w-q}^\bullet.
\]
\end{lem}
\begin{proof}
Since $R_1$ is a binary relation and the operators
$\mathscr{C}_{\ww{n}}^\bullet$, $\mathscr{BC}_q^{\bullet,m-1}$, and
$\mathscr{B}_{\ww{s}}^\bullet$ preserve binary relations (see \cite[Proposition~A.4]{CCM23}), the pair
\[
\mathscr{B}_{\ww{s}}^\bullet
\left(
\mathscr{BC}_q^{\bullet,m-1}
\left(
\mathscr{C}_{\ww{n}}^\bullet(R_1)
\right)
\right)
\]
is a binary relation. Suppose first that $\ww{s}\neq\varnothing$. By the definition of
$\mathscr{B}_{\ww{s}}^\bullet$, the second component of this binary
relation is zero. Thus, we obtain
\[
\mathscr{B}_{\ww{s}}^\bullet
\left(
\mathscr{BC}_q^{\bullet,m-1}
\left(
\mathscr{C}_{\ww{n}}^\bullet(R_1)
\right)
\right)
=
\left(
\sA^\bullet(\ww{s};m;\ww{n}),0
\right)
\]
and 
\[
\mathscr{L}_d^\bullet
\left(
\sA^\bullet(\ww{s};m;\ww{n})
\right)
=
0
\]
for all $d\in\ZZ$.
It follows immediately that
\[
\mathscr{L}_{\Acal_K}^\bullet
\left(
\sA^\bullet(\ww{s};m;\ww{n})
\right)
=
0.
\]

Now suppose that $\ww{s}=\varnothing$. Since
$\mathscr{B}_\varnothing^\bullet=\id$, write
\[
\mathscr{BC}_q^{\bullet,m-1}
\left(
\mathscr{C}_{\ww{n}}^\bullet(R_1)
\right)=(P,Q).
\]
By the definitions of $\mathscr{C}_{\ww{n}}^\bullet$ and
$\mathscr{BC}_q^{\bullet,m-1}$, the second component is
\[
Q
=
-L_1^m
\alpha_{1;(q-1)}^{\bullet,m}(\ww{n}),
\]
while
$
P+Q
=
\sA^\bullet(\varnothing;m;\ww{n})
$. Moreover, for every $d\geq0$, we have
\[
\mathscr{L}_d^\bullet(P)
+
\mathscr{L}_{d+1}^\bullet(Q)
=
0.
\]
Using
$P
=
\sA^\bullet(\varnothing;m;\ww{n})-Q$,
we obtain
\[
\mathscr{L}_d^\bullet
\left(
\sA^\bullet(\varnothing;m;\ww{n})
\right)
=
L_1^m
\left(
\mathscr{L}_{d+1}^\bullet
-
\mathscr{L}_d^\bullet
\right)
\left(
\alpha_{1;(q-1)}^{\bullet,m}(\ww{n})
\right).
\]
Notice that for $m\geq 1$, $\alpha_{1;(q-1)}^{\bullet,m}(\ww{n})$ has depth at least two, and hence \[
\mathscr{L}_{0}^\bullet(\alpha_{1;(q-1)}^{\bullet,m}(\ww{n}))=0.
\] 
For $D\geq 0$, summing over $0\leq d<D$ gives
\[
\mathscr{L}_{<D}^\bullet
\left(
\sA^\bullet(\varnothing;m;\ww{n})
\right)
=
L_1^m
\mathscr{L}_D^\bullet
\left(
\alpha_{1;(q-1)}^{\bullet,m}(\ww{n})
\right).
\]
Now, by the definition of $\alpha_{1;(q-1)}^\bullet$ and the properties of $\mathscr{L}_{<D}^{\bullet}$ (see \cite[Proposition~2.7]{CCM23}), we have
\[
\begin{aligned}
L_1^m
\mathscr{L}_D^\bullet
\left(
\alpha_{1;(q-1)}^{\bullet,m}(\ww{n})
\right)
=
\left(
L_1\mathscr{S}_D^\bullet(1)
\mathscr{L}_{<D}^\bullet([q-1])
\right)
L_1^{m-1}
\mathscr{L}_{<D}^\bullet
\left(
\alpha_{1;(q-1)}^{\bullet,m-1}(\ww{n})
\right).
\end{aligned}
\]

On the other hand, summing the Thakur's fundamental binary relation
\[
\mathscr{L}_d^\bullet([q])
-
L_1\mathscr{L}_{d+1}^\bullet([1,q-1])
=
0
\]
over $0\leq d<D$ yields
\[
L_1\mathscr{S}_D^\bullet(1)
\mathscr{L}_{<D}^\bullet([q-1])
=
\mathscr{L}_{<D}^\bullet([q])
-
L_1\mathscr{L}_{<D}^\bullet([1,q-1]).
\]
Consequently,
\[
\begin{aligned}
\mathscr{L}_{<D}^\bullet
\left(
\sA^\bullet(\varnothing;m;\ww{n})
\right)=
\left(
\mathscr{L}_{<D}^\bullet([q])
-
L_1\mathscr{L}_{<D}^\bullet([1,q-1])
\right)
L_1^{m-1}
\mathscr{L}_{<D}^\bullet
\left(
\alpha_{1;(q-1)}^{\bullet,m-1}(\ww{n})
\right).
\end{aligned}
\]
Therefore, it follows from the definitions that 
\[
\mathscr{L}_{\Acal_K}^\bullet
\left(
\sA^\bullet(\varnothing;m;\ww{n})
\right)
=
\gfk_{\Acal_K}
\mathscr{L}_{\Acal_K}^\bullet
\left(
L_1^{m-1}
\alpha_{1;(q-1)}^{\bullet,m-1}(\ww{n})
\right)\in
\gfk_{\Acal_K}
\Zcal_{\Acal_K,w-q}^\bullet
\]
for $w=mq+\wt(\ww{n})$.

Finally, let
$P\in\ker(\mathscr{L}_\infty^\bullet)$ be homogeneous of weight $w$.
By \cref{prop:Mishiba-kernel}, $P$ is a $K$-linear combination of
homogeneous elements
$\sA^\bullet(\ww{s};m;\ww{n})$ of weight $w$, and hence
\[
\mathscr{L}_{\Acal_K}^\bullet(P)
\in
\gfk_{\Acal_K}
\Zcal_{\Acal_K,w-q}^\bullet,
\]
as desired.
\end{proof}

Using this lemma, we show that $\zeta_{\Acal_K}(\ww{s})$ and $\Li_{\Acal_K,\ww{s}}(\mathbf{1})$ span the same space.
\begin{cor}\label{cor:finite-realization-span}
For every $\bullet\in\{\zeta,\Li\}$ and $w\geq0$, we have
\[
\Zcal_{\Acal_K,w}^\bullet
=
\Zcal_{\Acal_K,w}^{\bullet,\mathrm{T}}
+
\gfk_{\Acal_K}
\Zcal_{\Acal_K,w-q}^\bullet.
\]
Moreover, we have
\[
\Zcal_{\Acal_K,w}^\zeta
=
\Zcal_{\Acal_K,w}^\Li
=
\Zcal_{\Acal_K,w}
\]
for every $w\geq0$. In particular,
\[
\Zcal_{\Acal_K}^\zeta
=
\Zcal_{\Acal_K}^\Li
=
\Zcal_{\Acal_K}.
\]
\end{cor}

\begin{proof}
Fix $\bullet\in\{\zeta,\Li\}$ and $w\geq0$. Let
$\ww{s}\in\Ical$ be an index with $\wt(\ww{s})=w$. By
\cref{thm:basis_for_infty}, there exists
$
P_{\ww{s}}\in
\Span_K
\{
[\ww{n}]
:
\ww{n}\in\Ical_w^{\mathrm{T}}\}$
such that
\[
\mathscr{L}_\infty^\bullet([\ww{s}])
=
\mathscr{L}_\infty^\bullet(P_{\ww{s}}).
\]
Thus
$[\ww{s}]-P_{\ww{s}}
\in
\ker(\mathscr{L}_\infty^\bullet)$.
By \cref{lem:finite-CCM-relations},
\[
\mathscr{L}_{\Acal_K}^\bullet([\ww{s}])
-
\mathscr{L}_{\Acal_K}^\bullet(P_{\ww{s}})
\in
\gfk_{\Acal_K}
\Zcal_{\Acal_K,w-q}^\bullet.
\]
Thus, we obtain
\[
\Zcal_{\Acal_K,w}^\bullet
\subseteq
\Zcal_{\Acal_K,w}^{\bullet,\mathrm{T}}
+
\gfk_{\Acal_K}
\Zcal_{\Acal_K,w-q}^\bullet.
\]
Conversely, note that
\[
\Zcal_{\Acal_K,w}^{\bullet,\mathrm{T}}
\subseteq
\Zcal_{\Acal_K,w}^\bullet\qquad\text{and}\qquad
\gfk_{\Acal_K}
=
\mathscr{L}_{\Acal_K}^\bullet([q])
-
L_1\mathscr{L}_{\Acal_K}^\bullet([1,q-1])
\in
\Zcal_{\Acal_K,q}^\bullet.
\]
Hence,
\[
\Zcal_{\Acal_K,w}^\bullet
=
\Zcal_{\Acal_K,w}^{\bullet,\mathrm{T}}
+
\gfk_{\Acal_K}
\Zcal_{\Acal_K,w-q}^\bullet.
\]

Next, for every $\ww{s}\in\Ical^{\mathrm{T}}$, by \eqref{eq:Carlitz-formula}, we have 
$
\mathscr{L}_{\Acal_K}^\zeta([\ww{s}])
=
\mathscr{L}_{\Acal_K}^\Li([\ww{s}])$ and
therefore
\[
\Zcal_{\Acal_K,w}^{\zeta,\mathrm{T}}
=
\Zcal_{\Acal_K,w}^{\Li,\mathrm{T}}.
\]
It follows by induction on
$w$ that
\[
\Zcal_{\Acal_K,w}^\zeta
=
\Zcal_{\Acal_K,w}^\Li
\]
for every $w\geq0$, which completes the proof.
\end{proof}

We now construct the comparison map from $\infty$-adic multiple zeta values to a quotient of finite multiple zeta values.

\begin{thm}\label{thm:finite-comparison-map}
There exists a well-defined surjective
$K$-algebra homomorphism
\[
\widetilde{\Phi}_{\infty,\Acal_K}:
\Zcal_{\infty}
\longtwoheadrightarrow
{\Zcal_{\Acal_K}}
/{\gfk_{\Acal_K}\Zcal_{\Acal_K}}
\]
such that, for every $\bullet\in\{\zeta,\Li\}$ and
$\ww{s}\in\Ical$,
\[
\widetilde{\Phi}_{\infty,\Acal_K}
\left(
\mathscr{L}_\infty^\bullet([\ww{s}])
\right)
=
\mathscr{L}_{\Acal_K}^\bullet([\ww{s}])
\bmod
\gfk_{\Acal_K}\Zcal_{\Acal_K}.
\]
Moreover, we have
\[
\zeta_A(q-1)\Zcal_{\infty}
\subseteq
\ker\widetilde{\Phi}_{\infty,\Acal_K},
\]
and hence $\widetilde{\Phi}_{\infty,\Acal_K}$ induces a well-defined
surjective $K$-algebra homomorphism
\[
\Phi_{\infty,\Acal_K}:
{\Zcal_{\infty}}
/{\zeta_A(q-1)\Zcal_{\infty}}
\longrightarrow
{\Zcal_{\Acal_K}}
/{\gfk_{\Acal_K}\Zcal_{\Acal_K}}.
\]
\end{thm}
\begin{proof}
Fix $\bullet\in\{\zeta,\Li\}$. Define
\[
\widetilde{\Phi}_{\infty,\Acal_K}^\bullet:
\Zcal_{\infty}
\longrightarrow
{\Zcal_{\Acal_K}}
/{\gfk_{\Acal_K}\Zcal_{\Acal_K}}
\]
by
\[
\mathscr{L}_\infty^\bullet(P)
\longmapsto
\mathscr{L}_{\Acal_K}^\bullet(P)
\bmod
\gfk_{\Acal_K}\Zcal_{\Acal_K}.
\]
We first show that this map is well-defined. Suppose that
$P,Q\in\Hfk$ satisfy
\[
\mathscr{L}_\infty^\bullet(P)
=
\mathscr{L}_\infty^\bullet(Q).
\]
Then $P-Q\in\ker(\mathscr{L}_\infty^\bullet)$.
By
\cref{prop:Mishiba-kernel,lem:finite-CCM-relations}, it follows that
\[
\mathscr{L}_{\Acal_K}^\bullet(P)
\equiv
\mathscr{L}_{\Acal_K}^\bullet(Q)
\pmod{\gfk_{\Acal_K}\Zcal_{\Acal_K}},
\]
so $\widetilde{\Phi}_{\infty,\Acal_K}^\bullet$ is well-defined. Moreover, \cref{thm:strict-bullet-character,thm:finite-character} show that
$\widetilde{\Phi}_{\infty,\Acal_K}^\bullet$ is a $K$-algebra homomorphism and the surjectivity follows from \cref{cor:finite-realization-span}.

Finally, for every $\ww{s}\in\Ical^{\mathrm{T}}$, by \eqref{eq:Carlitz-formula}, we have
\[
\mathscr{L}_\infty^\zeta([\ww{s}])
=
\mathscr{L}_\infty^\Li([\ww{s}]),
\qquad
\mathscr{L}_{\Acal_K}^\zeta([\ww{s}])
=
\mathscr{L}_{\Acal_K}^\Li([\ww{s}]).
\]
By \cref{thm:basis_for_infty}, these elements span
$\Zcal_{\infty}$ over $K$. Hence
\[
\widetilde{\Phi}_{\infty,\Acal_K}^\zeta
=
\widetilde{\Phi}_{\infty,\Acal_K}^\Li.
\]
We denote this common map by
$\widetilde{\Phi}_{\infty,\Acal_K}$. Since
\[
\widetilde{\Phi}_{\infty,\Acal_K}
\left(\zeta_A(q-1)\right)
=
\zeta_{\Acal_K}(q-1)
=
0
\]
by \cite[Proposition~4.1.3(1)]{Shi18}, we have
\[
\zeta_A(q-1)\Zcal_{\infty}
\subseteq
\ker\widetilde{\Phi}_{\infty,\Acal_K},
\]
which completes the proof.
\end{proof}

To establish the compatibility of the comparison map with the $\star$-inverse values, we first consider the images of the finite realizations.

\begin{lem}\label{lem:finite-star-inverse-space}
Let
$\bullet\in\{\zeta,\Li\}$. Then
\[
\mathscr{L}_{\Acal_K}^{\bullet,\ant}(\Hfk)
=
\mathscr{L}_{\Acal_K}^{\bullet}(\Hfk)
=
\Zcal_{\Acal_K}.
\]
\end{lem}

\begin{proof}
We apply \cref{lem:image-comparison} to the case
\[
B=\Acal_K,
\qquad
\ev=\mathscr{L}_{\Acal_K}^\bullet,
\qquad
\ev^\ant=\mathscr{L}_{\Acal_K}^{\bullet,\ant}.
\]
By \cref{lem:degree-convolution}, Conditions
\ref{cond:unital} and \ref{cond:convolution} hold, while \cref{thm:finite-character} gives Condition
\ref{cond:ev-product}. We next verify Condition \ref{cond:vanishing}. If $\bullet=\zeta$ and
$q-1\mid j$, then
\[
\ev^\ant([j])
=
-\zeta_{\Acal_K}(j)
=
0
\]
by \cite[Proposition~4.1.3(1)]{Shi18}.
If $\bullet=\Li$, then
\[
\ev^\ant([q-1])
=
-\Li_{\Acal_K,q-1}(\mathbf{1})
=
-\zeta_{\Acal_K}(q-1)
=
0.
\]
Hence, \cref{lem:bullet-stuffle-transfer} gives Condition
\ref{cond:ant-product}. Therefore, by \cref{lem:image-comparison} and \cref{cor:finite-realization-span}, we obtain
\[
\mathscr{L}_{\Acal_K}^{\bullet,\ant}(\Hfk)
=
\mathscr{L}_{\Acal_K}^{\bullet}(\Hfk)=\Zcal_{\Acal_K}.
\]
\end{proof}

For the $\infty$-adic realizations, we also
have (see \cite[Definition~1.3.1 and
Proposition~2.2.4]{Mis26}; see also
\cite[\S4.1]{TsuCoaction26})
\begin{equation}\label{lem:infty-star-inverse-space}
\mathscr{L}_\infty^{\bullet,\ant}(\Hfk)
\subseteq
\mathscr{L}_\infty^\bullet(\Hfk)
=
\Zcal_{\infty}
\end{equation}
for $\bullet\in\{\zeta,\Li\}$.
Then we show that the comparison map is compatible with the $\star$-inverse values.

\begin{thm}\label{thm:finite-star-inverse-comparison}
For every
$\bullet\in\{\zeta,\Li\}$ and $\ww{s}\in\Ical$, we have
\[
\widetilde{\Phi}_{\infty,\Acal_K}
\left(
\mathscr{L}_\infty^{\bullet,\ant}([\ww{s}])
\right)
=
\mathscr{L}_{\Acal_K}^{\bullet,\ant}([\ww{s}])
\bmod
\gfk_{\Acal_K}\Zcal_{\Acal_K}.
\]
\end{thm}
\begin{proof}
Fix $\bullet\in\{\zeta,\Li\}$. By \cref{thm:finite-comparison-map},
$\widetilde{\Phi}_{\infty,\Acal_K}$ factors through the quotient
$\Zcal_{\infty}/\zeta_A(q-1)\Zcal_{\infty}$.
Hence it suffices to show that
\[
\Phi_{\infty,\Acal_K}
\left(
\mathscr{L}_\infty^{\bullet,\ant}([\ww{s}])
\bmod \zeta_A(q-1)\Zcal_{\infty}
\right)
=
\mathscr{L}_{\Acal_K}^{\bullet,\ant}([\ww{s}])
\bmod \gfk_{\Acal_K}\Zcal_{\Acal_K}.
\] 
We now put
\[
B={\Zcal_{\infty}}
/{\zeta_A(q-1)\Zcal_{\infty}}
\times
{\Zcal_{\Acal_K}}
/{\gfk_{\Acal_K}\Zcal_{\Acal_K}}.
\]
Define $K$-linear maps
$\ev,\ev^\ant:\Hfk\longrightarrow B$
by
\[
\ev(P)
=
\left(
\mathscr{L}_\infty^\bullet(P)
\bmod \zeta_A(q-1)\Zcal_{\infty},
\mathscr{L}_{\Acal_K}^\bullet(P)
\bmod \gfk_{\Acal_K}\Zcal_{\Acal_K}\right)
\]
and
\[
\ev^\ant(P)
=
\left(
\mathscr{L}_\infty^{\bullet,\ant}(P)
\bmod \zeta_A(q-1)\Zcal_{\infty},
\mathscr{L}_{\Acal_K}^{\bullet,\ant}(P)
\bmod \gfk_{\Acal_K}\Zcal_{\Acal_K}\right).
\]
This is well-defined by \eqref{lem:infty-star-inverse-space} and \cref{lem:finite-star-inverse-space}. By \cref{lem:degree-convolution}, Conditions
\ref{cond:unital} and \ref{cond:convolution} hold componentwise, while \cref{thm:strict-bullet-character,thm:finite-character} give Condition
\ref{cond:ev-product}. 

Moreover, Condition \ref{cond:vanishing} holds for the $\infty$-adic
component by \cite[\S4.1]{TsuCoaction26}, and for the finite component by
\cite[Proposition~4.1.3(1)]{Shi18}.
Hence \cref{lem:bullet-stuffle-transfer} gives Condition
\ref{cond:ant-product}. Therefore, by \cref{lem:image-comparison},
\[
\ev(\Hfk)=\ev^\ant(\Hfk).
\]
In particular, for every $\ww{s}\in\Ical$, there exists
$P\in\Hfk$ such that
\[
\mathscr{L}_\infty^{\bullet,\ant}([\ww{s}])
\equiv
\mathscr{L}_\infty^\bullet(P)
\pmod{\zeta_A(q-1)\Zcal_{\infty}}
\]
and
\[
\mathscr{L}_{\Acal_K}^{\bullet,\ant}([\ww{s}])
\equiv
\mathscr{L}_{\Acal_K}^\bullet(P) \pmod{ \gfk_{\Acal_K}\Zcal_{\Acal_K}}.
\]
Thus, by \cref{thm:finite-comparison-map},
\[
\begin{aligned}
&
\Phi_{\infty,\Acal_K}
\left(
\mathscr{L}_\infty^{\bullet,\ant}([\ww{s}])
\bmod \zeta_A(q-1)\Zcal_{\infty}
\right)
\\
&\qquad=
\Phi_{\infty,\Acal_K}
\left(
\mathscr{L}_\infty^\bullet(P)
\bmod \zeta_A(q-1)\Zcal_{\infty}
\right)
\\
&\qquad=
\mathscr{L}_{\Acal_K}^\bullet(P)
\bmod \gfk_{\Acal_K}\Zcal_{\Acal_K}
\\
&\qquad=
\mathscr{L}_{\Acal_K}^{\bullet,\ant}([\ww{s}])
\bmod \gfk_{\Acal_K}\Zcal_{\Acal_K},
\end{aligned}
\]
as desired.
\end{proof}

\section{Dimension Formulas and a Comparison Conjecture}\label{sec:dimension}
In this section, we compare the dimensions of the spaces of
$\infty$-adic, $v$-adic, and finite multiple zeta values.
In particular, we use the comparison map constructed in \cref{thm:finite-comparison-map} to establish the upper bound
predicted by Shi's dimension conjecture.
We conclude by formulating a unified comparison conjecture.
\subsection{Dimension Formulas}
We fix a finite place $v$ of $K$ and consider the notion of \textit{$v$-adic multiple zeta values}, denoted by $\zeta_A(\ww{s})_v$ (see \cite[Definition 6.1.1]{CM21}).  
We write $\Zcal_v$ for the $K$-vector space generated by all $v$-adic multiple zeta values and $\Zcal_{v,w}$ for the $K$-vector space generated by those of weight $w\geq 0$. We refer to
\cite{CM19,CM21,CCM22} for details.

We define the following three power series:
\[
\begin{aligned}
H_\infty(X)
&=
\frac{1-X^q}
{1-X-X^2-\cdots-X^q}
=
\sum_{w\geq0}D_{\infty,w}X^w,\\
H_v(X)
&=
\frac{(1-X^{q-1})(1-X^q)}
{1-X-X^2-\cdots-X^q}
=
\sum_{w\geq0}D_{v,w}X^w,\\
H_{\Acal_K}(X)
&=
\frac{1-X^{q-1}}
{1-X-X^2-\cdots-X^q}
=
\sum_{w\geq0}D_{\Acal_K,w}X^w.
\end{aligned}
\]
One checks that
\[
D_{\infty,w}
=
\begin{cases}
1, & w=0,\\[1mm]
2^{w-1}, & 1\leq w<q,\\[1mm]
2^{w-1}-1, & w=q,\\[1mm]
\displaystyle\sum_{i=1}^q D_{\infty,w-i}, & w>q,
\end{cases}
\]
\[
D_{v,w}
=
\begin{cases}
1, & w=0,\\[1mm]
2^{w-1}, & 1\leq w<q-1,\\[1mm]
2^{q-2}-1, & w=q-1,\\[1mm]
2^{q-1}-2, & w=q,\\[1mm]
\displaystyle\sum_{i=1}^q D_{v,w-i}+1, & w=2q-1,\\[1mm]
\displaystyle\sum_{i=1}^q D_{v,w-i}, & w>q,\ w\neq 2q-1,
\end{cases}
\]
and
\[
D_{\Acal_K,w}
=
\begin{cases}
1, & w=0,\\[1mm]
2^{w-1}, & 1\leq w<q-1,\\[1mm]
2^{w-1}-1, & w=q-1,q,\\[1mm]
\displaystyle\sum_{i=1}^q D_{\Acal_K,w-i}, & w>q.
\end{cases}
\]
These three sequences are closely related to the dimensions of the
corresponding spaces of multiple zeta values.

For $\infty$-adic multiple zeta values, Todd conjectured
\cite[Conjecture~7.1]{Tod18} that
\[
\dim_K\Zcal_{\infty,w}
=
D_{\infty,w}=\#\Ical_w^{\mathrm{T}}
\qquad(w\geq0),
\]
which now follows from \cref{thm:basis_for_infty}. 

For $v$-adic multiple zeta values, Chang, Chen, and Mishiba \cite[Theorem 1.2.3]{CCM22} (see also \cite[Theorem 6.4.1]{CM19}) constructed a surjective $K$-algebra homomorphism
\[
\Phi_{\infty,v}:
\Zcal_{\infty}/\zeta_A(q-1)\Zcal_{\infty}
\longtwoheadrightarrow
\Zcal_v
\]
satisfying
\[
\Phi_{\infty,v}
\left(
\zeta_A(\ww{s})
\bmod \zeta_A(q-1)\Zcal_{\infty}
\right)
=
\zeta_A(\ww{s})_v
\]
for every $\ww{s}\in\Ical$. They further conjectured that this map is an isomorphism.

\begin{conj}[{\cite[Conjecture~5.4.1]{CCM22}}]\label{conj:CCM-v-adic}
The map
\[
\Phi_{\infty,v}:
\Zcal_{\infty}/\zeta_A(q-1)\Zcal_{\infty}
\xrightarrow{\ \sim\ }
\Zcal_v
\]
is an isomorphism of $K$-algebras.
\end{conj}

In particular, this conjectural isomorphism predicts that
\[
\dim_K\Zcal_{v,w}
=
D_{v,w}
\qquad (w\geq0),
\]
while the corresponding upper bound
\[
\dim_K\Zcal_{v,w}
\leq
D_{v,w}
\qquad (w\geq0)
\]
was established in \cite[Corollary~1.8]{CCM23}.

For finite multiple zeta values over $K$, Shi \cite{Shi18} proposed, based on explicit
computations of relations among truncated multiple zeta values, the
following dimension conjecture.

\begin{conj}[{\cite[Conjecture~4.6.1]{Shi18}}]\label{conj:Shi}
For every $w\geq0$, one has
\[
\dim_K\Zcal_{\Acal_K,w}
=
D_{\Acal_K,w}.
\]
\end{conj}

Using the comparison map constructed in
\cref{thm:finite-comparison-map}, we obtain the corresponding upper bound.

\begin{thm}\label{thm:finite-dimension-upper-bound}
For every $w\geq0$, we have
\[
\dim_K\Zcal_{\Acal_K,w}
\leq
D_{\Acal_K,w}.
\]
\end{thm}

\begin{proof}
Put $\Zcal_{\infty,w}=\Zcal_{\Acal_K,w}=0$ for $w<0$. By the construction in the proof of
\cref{thm:finite-comparison-map}, together with
\cref{lem:finite-CCM-relations}, there exists a surjective $K$-linear map
\[
{\Zcal_{\infty,w}}
/{\zeta_A(q-1)\Zcal_{\infty,w-(q-1)}}
\longrightarrow
{\Zcal_{\Acal_K,w}}
/{\gfk_{\Acal_K}\Zcal_{\Acal_K,w-q}}.
\]
Since $\zeta_A(q-1)\neq0$, we have
\[
\dim_K
{\Zcal_{\infty,w}}
/{\zeta_A(q-1)\Zcal_{\infty,w-(q-1)}}
=
D_{\infty,w}-D_{\infty,w-(q-1)}
=
D_{v,w}.
\]
Thus, we obtain
\[
\begin{aligned}
\dim_K\Zcal_{\Acal_K,w}
&=
\dim_K
{\Zcal_{\Acal_K,w}}
/{\gfk_{\Acal_K}\Zcal_{\Acal_K,w-q}}
+
\dim_K
\left(
\gfk_{\Acal_K}\Zcal_{\Acal_K,w-q}
\right)\\
&\leq
D_{v,w}+\dim_K\Zcal_{\Acal_K,w-q}.
\end{aligned}
\]
Since $H_v(X)=(1-X^q)H_{\Acal_K}(X)$, we have
\[
D_{v,w}
=
D_{\Acal_K,w}-D_{\Acal_K,w-q},
\]
where $D_{\Acal_K,w-q}=0$ if $w<q$.
We prove the result by induction on $w$. For $w<q$, the assertion is
immediate. For $w\geq q$, the induction hypothesis gives
\[
\dim_K\Zcal_{\Acal_K,w}
\leq
D_{\Acal_K,w}-D_{\Acal_K,w-q}
+D_{\Acal_K,w-q}
=
D_{\Acal_K,w}.
\]
\end{proof}

\begin{rmk}\label{rmk:gisunit}
In fact, by Thakur's fundamental binary relation (see \eqref{eq:R_1}), we have
\[
\gfk_{\Acal_K}
=
(S_{\deg v-1}(q)\bmod{v})_v
=
\left(\frac{1}{L_{\deg v-1}^q}\bmod{v}\right)_v.
\]
Note that since $v\nmid L_{\deg v-1}$,
$\gfk_{\Acal_K}\in\Acal_K^\times$.
In particular, multiplication by $\gfk_{\Acal_K}$ is injective on
$\Zcal_{\Acal_K}$, and hence
\[
\dim_K
\left(
\gfk_{\Acal_K}\Zcal_{\Acal_K,w-q}
\right)
=
\dim_K\Zcal_{\Acal_K,w-q}
\]
for every $w\geq0$.
\end{rmk}

\subsection{A Unified Comparison Conjecture}

Combining \cref{conj:ffKZ,conj:CCM-v-adic}, we are led to the following unified comparison conjecture.

\begin{conj}\label{conj:unified-comparison}
The comparison maps
\[
\Zcal_v
\xleftarrow[\ \sim\ ]{\Phi_{\infty,v}}
{\Zcal_{\infty}}
/{\zeta_A(q-1)\Zcal_{\infty}}
\xrightarrow[\ \sim\ ]{\Phi_{\infty,\Acal_K}}
{\Zcal_{\Acal_K}}
/{\gfk_{\Acal_K}\Zcal_{\Acal_K}}
\]
are $K$-algebra isomorphisms.
\end{conj}

The author plans to investigate this conjecture in future work.

\section*{Acknowledgments}
The author is grateful to Chieh-Yu Chang for carefully reading the manuscript and for valuable suggestions that
improved its presentation. This work was carried out during the author's
visit to Tohoku University in Sendai, which was supported by the
Japan-Taiwan Exchange Association. The author also gratefully
acknowledges the support of the National Science and Technology Council
over the past few years under grant no.\ 113-2628-M-007-004.

%%%%%%%%%%%%%%%%%%%%%%%%%%%%%%%%%%%%%%%%
%%%%%%%%%%%%%%%%%%%%%%%%%%%%%%%%%%%%%%%%
\sloppy

\end{document}